\documentclass{amsart}
\usepackage{amsmath,amssymb}%,bbm,latexsym}
\usepackage{amsfonts}
\usepackage{color}
\usepackage[left=3.5cm,right=3.5cm, top=4cm,bottom=4cm]{geometry}
\usepackage{graphicx}
\newtheorem{theorem}{Theorem}[section]
\newtheorem{lemma}[theorem]{Lemma}
\newtheorem{proposition}[theorem]{Proposition}
\newtheorem{corollary}[theorem]{Corollary} 
 \theoremstyle{remark}
\newtheorem{remark}[theorem]{Remark}

\usepackage{graphicx}
\usepackage{subcaption}

\newcommand{\E}{\mathbb{E}}

\newcommand{ \blue   }{\color{black}}
\newcommand{  \red  }{\color{black}}
\newcommand{ \black }{\color{black} }

\newcommand{ \kat}{\color{black} }

\begin{document}

\title[Randomized exponential integrator for modulated NLS]{Randomized exponential integrators for modulated nonlinear Schr\"odinger equations}

\author{Martina Hofmanov\'a}
\address[M. Hofmanov\'a]{Fakult\"at f\"ur Mathematik, Universit\"at Bielefeld, Postfach 10 01 31, 33501 Bielefeld, Germany}
\email{hofmanova@math.uni-bielefeld.de}

\author{Marvin Kn\"oller}
\address{Fakult\"{a}t f\"{u}r Mathematik, Karlsruhe Institute of Technology,
Englerstr. 2, 76131 Karlsruhe, Germany}
\email{marvin.knoeller@kit.com}

\author{Katharina Schratz}
\address{Fakult\"{a}t f\"{u}r Mathematik, Karlsruhe Institute of Technology,
Englerstr. 2, 76131 Karlsruhe, Germany}
\email{katharina.schratz@kit.edu}

\thanks{}

\begin{abstract}
We consider the nonlinear Schr\"odinger equation with dispersion modulated by a (formal) derivative of a time-dependent function  with fractional Sobolev regularity of class $W^{\alpha,2}$ for some $\alpha\in (0,1)$. Due to the  loss of smoothness in  the problem  classical numerical methods face severe order reduction.
In this work, we develop and analyze a new randomized  exponential integrator  based on a stratified Monte Carlo approximation. The new discretization technique averages the high oscillations in the solution allowing for   improved  convergence rates  of order $\alpha+1/2$.  In addition, the new approach allows us to treat a far more general class of modulations than  the available literature.  Numerical results underline our theoretical findings and show the favorable error behavior of our new scheme compared to classical methods.
\end{abstract}

\subjclass[2010]{}

\keywords{}

\date{\today}

\maketitle

\section{Introduction}

We study  modulated nonlinear Schr\"odinger equations of the form
\begin{equation}\label{mnls}
 i \partial_t u(t,x) + \Delta u(t,x) \partial_t g(t) = \left| u(t,x)\right |^2 u(t,x), \quad u(0,x) = u_0(x),\quad (t,x) \in \mathbb{R} \times \mathbb{T}^d,
\end{equation}
where $g:\mathbb{R}\to\mathbb{R}$ is an arbitrary function of time  with fractional Sobolev regularity of class $W^{\alpha,2}$ (see \eqref{sobi} for the definition) with $\alpha\in (0,1)$.

Classical semilinear Schr\"odinger equations with $g(t) = t$ are nowadays extensively studied  numerically. In this context, splitting methods (where the right-hand side is split into the kinetic and nonlinear part, respectively) as well as exponential integrators (based on approximating Duhamel's formula) contribute particularly attractive classes of integration schemes. For an extensive overview on splitting and exponential integration methods we refer to \cite{HLW,HochOst10,HLRS10,McLacQ02}, and for their rigorous convergence analysis in the context of semilinear Schr\"odinger equations we refer to \cite{BeDe02,CCO08,CoGa12,Duj09,Faou12,Gau11,Lubich08} and the references therein.

In the last decades, the modulated Schr\"odinger equation \eqref{mnls} has  gained a lot of attention in physics  serving, e.g.,  as a model describing  propagation of light waves in optical dispersion-managed fibers (see for instance \cite{ABP00,GC15, dBD10,DT,HL09,HL12,KMZ, ZGJT}). Numerically, however, only very little is known so far for this type of problem.  The highly oscillatory nature of the problem makes  the construction and analysis of numerical schemes for~\eqref{mnls} particularly challenging in  non-smooth regimes, where $g$ is   of fractional Sobolev regularity of class $W^{\alpha,2}$  with exponent $\alpha \in(0, 1)$. Due to the loss of smoothness in the system  classical numerical schemes (e.g., splitting methods, exponential integrators, etc.)  face severe order reduction. This is due to the fact that their rate of convergence heavily depends on the smoothness of $g$.  To allow a reliable approximation classical numerical schemes are thus subject to severe step size restrictions which leads to large errors and huge computational costs in case of  non-smooth modulations $g$. We refer to \cite{EFHI09,HLW} for an extensive overview on the numerical analysis of highly oscillatory problems, and in particular to \cite{iserles,jin} and the references therein for the numerical analysis of semiclassical Schr\"odinger equations.    Recently, the case of $g$ being a Brownian motion,  a stochastic process that has in particular $\alpha$-H\"older continuous trajectories for every $\alpha\in (0,1/2)$,   has  gained a lot attention in numerical analysis. Various numerical schemes have been proposed reaching from  Crank-Nicolson discretizations \cite{BdBD} over splitting schemes \cite{M} up to exponential integrators \cite{CD}. A rigorous error analysis of these schemes could be established based on probabilistic arguments, i.e., stating the `rate of convergence with certain probability' (see \eqref{prob} below for the detailed convergence bound). Whereas the used probabilistic arguments are applicable to Brownian motion they do not allow for an extension to more general modulations $g$  nor  for deterministic error bounds.
%A generalization appeared in \cite{GC15}, where an Euler-like procedure  was used to construct solutions under certain irregularity assumption on $g$, treating for instance the case of a fractional Brownian motion. Numerical analysis was not the objective of this work since the approximation procedure assumes  existence of certain integral operators, which were not resolved numerically. \rmk{is this ok?}
 Up to our knowledge, nothing is known  in the  numerical analysis literature so far for a general modulated dispersion $ \partial_t g(t) \Delta $  and, in particular, deterministic (pathwise) error bounds are still lacking.\black
 
  In this work we consider for the first time the general case where the modulation $g$ is  of fractional Sobolev regularity of class $W^{\alpha,2}$  in time  and develop a new randomized exponential integration scheme for the modulated Schr\"odinger equation \eqref{mnls} in this general setting.  For technical reasons, we also require that $g$ belongs to the H\"older space $C^{\varepsilon}$ for an arbitrarily small $\varepsilon\in (0,1)$. \blue This allows us to prove the a priori boundedness of the numerical solution in the classical Sobolev space $H^\sigma$ (cf. \cite{Lubich08}). Note
  %\red This is only used in order to obtain a priori boundedness of the numerical solution, which in turn influences the proportionality constant of our error bound and the maximal step size.  However, the established order of convergence is  independent of $\varepsilon$.    \black Moreover, 
\black  that if $\alpha>1/2$ then according to  the Sobolev embedding there exists $\varepsilon\in (0,1)$ such that $W^{\alpha,2}\subset C^{\varepsilon}$, so the additional assumption does not pose any further restriction.

 Our new technique allows us to average the high oscillations in the solution and gain $1/2$ in the convergence rate. More precisely, we can firstly   establish    improved pathwise  convergence rates of order $\alpha+1/2$  for the new scheme  (see \eqref{errK} below for the precise error estimate).  In contrast, classical schemes \blue require much stronger smoothness assumptions as they rely
  on the H\"older continuity of $g$. Even if  $g$ is $\alpha$-H\"older continuous (and not only in $W^{\alpha,2}$) classical schemes face severe order reduction down to order $\alpha$ due to the loss of smoothness in the system. These theoretical findings are underlined by numerical experiments: while the error of classical schemes oscillates widely, reaching large errors, our proposed randomized exponential integrator allows us to average these oscillations maintaining its convergence rate with much smaller  and reliable errors, see Section~\ref{sec:num}.

%We propose a new scheme  based on a stratified Monte-Carlo approximation and prove its convergence of order $\alpha+1/2$ for $g$ $\alpha$-H\"older continuous. Numerical experiments underline the favorable error behavior of our new scheme compared to classical methods.

For completeness, note that randomized numerical schemes in various settings have already appeared in the literature,  in particular in the context of ordinary differential equations and recently also in case of uncertainty quantifications \cite{Assyr}\black . Let us particularly mention \cite{EKKL,JN09, KW, KWa} which inspired our research and where further references can be found.  However, to the best of our knowledge such a method was not yet applied in the context of dispersive equations such as \eqref{mnls}. In \cite{GC15}, the modulated Schr\"odinger equation of type \eqref{mnls} has been studied from the analytical point of view. Thereby an Euler-like procedure was used to prove existence and construct solutions under certain regularity assumption on $g$. However, numerical analysis was not the objective of this work.

\textbf{New approach in a nutshell.} In the present paper we put forward an exponential integrator for \eqref{mnls} based on a stratified Monte Carlo approximation. More precisely, we consider the mild formulation of \eqref{mnls} and approximate the convolution integral appearing on the right hand side by means of a randomized Riemann sum. The key idea is to choose the randomization in such a way, that the associated error is a (discrete time) martingale, which permits to apply the so-called Burkholder-Davis-Gundy inequality (see Theorem \ref{bdg}). Remarkably, this allows to gain half in the convergence rate in expectation. More precisely, splitting the time interval $[0, T]$ into an equidistant partition $t_0, \ldots, t_N$ with the mesh
size $\tau = \frac{T}{N}$, for some $N\in\mathbb{N}$, we establish the following bound for the difference of the exact solution $u$ at time $t_{n}$ and the numerical solution  $u^n$ valid for initial conditions in $H^{\sigma+2}$ with $\sigma>d/2$ (see Theorem \ref{glob} for details):
\begin{equation}\label{errK}
\left( \E_\xi\max_{n=0,\dots,N}\|u(t_n)-u^n\|_{H^{\sigma}}^2\right)^{1/2} \leq c \tau^{\min\{1,{\alpha+\frac{1}{2}}\}},
\end{equation}
where $\E_\xi$ is the expected value associated to the randomization of the numerical scheme  and $\alpha$ is the exponent of   fractional Sobolev regularity \black of $g$. Classical numerical schemes  are in contrast restricted to the H\"older continuity, see also Remark \ref{rem:classic} below.\black

\textbf{Comparison with previous results.} We stress that the expectation above is taken only with respect to the randomization of the scheme, whereas the modulation $g$ is deterministic. In other words, in the case of a random modulation, such as the Brownian motion treated in \cite{BdBD,CD}, our result provides  {\em pathwise} error estimates leading to a {\em pathwise} convergence analysis. In addition, it applies to a wide range of possible modulations, deterministic or random, and in particular to any other stochastic process with  trajectories of fractional Sobolev regularity, overcoming the limitations of \cite{GC15}.

 Note that in \cite{BdBD}  a semi-discrete Crank-Nicolson scheme is analyzed. The strong order of convergence in
probability in $H^\sigma$, $\sigma>d/2$, is equal to one for initial conditions in $H^\sigma$ for this scheme. In \cite{CD} an explicit exponential integrator is proposed and it is shown that it is mean-square order 1 in $H^\sigma$ for initial conditions in $H^{\sigma+4}$.

In both works \cite{BdBD,CD}, a stochastic approach was used and it was only possible to  obtain the order of convergence $1$ in probability which reads as
\begin{align}\label{prob}
\lim_{C\to\infty}\mathbb{P}_W\big( \|u(t_n)-u^n\|_\sigma>C\tau\big) =0,
\end{align}
for all $n=0,\dots,N$ uniformly with respect to $\tau$. Here we write  $\mathbb{P}_W$ to denote the probability measure on the probability space where the Brownian motion $g=W$ is defined (in contrast to the notation $\E_\xi$ above which only concerns the artificial randomness introduced in our scheme). It was conjectured in \cite{BdBD} that a stronger result should hold, namely,
$$
\lim_{C\to\infty}\mathbb{P}_W\Big(\max_{n=0,\dots,N} \|u(t_n)-u^n\|_\sigma>C\tau\Big) =0,
$$
uniformly with respect to $\tau$. However, the proof would require  very tedious computations which were not presented.
Our approach allows to overcome this challenge and the convergence analysis is rather elegant and simple. See Remark \ref{rem:12} below.

%Note that since the Brownian motion $W$ is $\alpha$-H\"older continuous for $\alpha\in(0,1/2)$ our analysis gives convergence rate $\alpha+1/2$. On the other hand, exploiting instead the fact that
%$$
%\E_W[|W(t)-W(s)|^2]=|t-s|,\qquad s,t\in [0,\infty),
%$$
%together with the independence of the increments, our analysis would lead to the same order of convergence as in \cite{BdBD,CD}, namely 1.

\textbf{Organization of the paper.} 
Our exponential integrator is derived in Section \ref{sec:deriv}.  In Section~\ref{sec:conv} we carry out a rigorous  error analysis for the new scheme.  \kat In Section \ref{sec:num} we present numerical experiments for $g$ belonging to  $W^{\alpha,2}$ for various $\alpha \in (0,1)$.

For practical implementation issues we impose periodic boundary conditions, i.e, $x \in \mathbb{T}^d= [0,2\pi]^d$, $d\in\mathbb{N}$. For notational simplicity we work with cubic nonlinearities. However, the generalization to  polynomial nonlinearities $f(u) = \vert u\vert^{2p} u$ is straightforward.

\bigskip

\textbf{Notation.} 
In the following let $\sigma>d/2$. We denote by $\Vert \cdot \Vert_\sigma$ the standard $H^\sigma= H^\sigma(\mathbb{T}^d)$ Sobolev norm, that is, if $f\in H^{\sigma}$ such that $f(x)= \sum_{k \in \mathbb{Z}^d} \hat{f}_k \mathrm{e}^{i k\cdot x}$, $x\in \mathbb{T}^{d}$, then
$$
\|f\|_{\sigma}=\left(\sum_{k\in\mathbb{Z}^{d}}(1+|k|^{2})^{\sigma}|\hat f_{k}|^{2}\right)^{1/2}.
$$
In particular, we exploit the well-known bilinear estimate
\begin{equation}\label{bil}
\Vert f_1 f_2 \Vert_\sigma \leq c_{\sigma} \Vert f_1 \Vert_\sigma \Vert f_2 \Vert_\sigma
\end{equation}
which holds for some constant $c_{\sigma}>0$ and every $f_{1},f_{2}\in H^{\sigma}$.
The space of $\varepsilon$-H\"older continuous functions for $\varepsilon\in (0,1)$ is denoted by $C^{\varepsilon}$, whereas
$C([0,T];H^{\sigma})$ denotes the space of continuous functions with values in $H^{\sigma}$.
 Furthermore, for $\alpha\in (0,1)$ we let $W^{\alpha,2}$ denote  the fractional Sobolev space on $(0,T)$ (also called the Sobolev-Slobodeckij space) given by the norm
\begin{equation}\label{sobi}
\Vert g \Vert_{W^{\alpha,2}} = \left(\int_0^T \vert g(t)\vert^2 \mathrm{d}t + \int_0^T \int_0^T 
\frac{\vert g(r)-g(s)\vert^2}{\vert r - s \vert ^{2\alpha+1}}\mathrm{d}r\mathrm{d}s\right)^\frac{1}{2}.
\end{equation}\black

%The space of $\alpha$-H\"older continuous functions for $\alpha\in (0,1)$ is denoted by $C^{\alpha}([0,T])$, whereas $C([0,T];H^{\sigma})$ denotes the space of continuous functions with values in $H^{\sigma}$.

For two quantities $a,b$, we use the notation $a\lesssim b$ to say that there exists a constant $c>0$ such that $a\leq cb$. This proportionality constant will typically depend on data of the problem and/or certain norms of the exact solution, which we always specify below. However, it will always be independent of the step size $\tau$ and $N$.

%In the sequel, we split a time interval  $[0, T]$ for $T>0$ into an equidistant partition $t_0, \ldots, t_N$, with the mesh
%size $\tau = \frac{T}{N}$, for some $N\in \mathbb{N}$.

\section{Derivation of the numerical scheme}
\label{sec:deriv}

It is possible to make sense of \eqref{mnls} using the associated mild formulation. To this end, let us set
\[S(t) = S^g(t) = e^{i g (t) \partial^2_x},\quad t\in\mathbb{R},\]
as well as
\[U (t, r) = S(t) S(r)^{-
1} = e^{i [g (t) - g (r)] \partial^2_x}\]
such that in particular $U(t,0) = S(t)$. Note that due to the presence of the modulation $g$, the problem \eqref{mnls} is not time-homogeneous. Consequently, the linear part of \eqref{mnls} generates an evolution system given by the family of operators $U(t,r) $, $r,t\in\mathbb{R}$, which are (generally) not functions of the difference $t-r$ as it would be case in the classical setting, that is, $g(t)=t$. Intuitively, $U(t,r)$ describes the evolution of the linear part of \eqref{mnls} from time $r$ to time $t$. Note that since the modulation $g$ does not depend on the space variable $x$, the above operators are  Fourier multipliers given by 
$$\mathcal{F}[S(t)f](k)=e^{-i g(t)|k|^2}\hat{f}_k,\qquad \mathcal{F}[U(t,r)f](k)=e^{-i [g(t)-g(r)]|k|^2}\hat{f}_k$$
for $f\in L^{2}$ such that  $f(x) = \sum_{k \in \mathbb{Z}^d} \hat{f}_k \mathrm{e}^{i k\cdot x}$, $x\in \mathbb{T}^{d}$. 

In the construction and analysis of our numerical scheme we will in particular exploit that  the operators $U(t,r)$ and $S(t)$ are linear isometries in  $H^\sigma$ for all $r, t \in \mathbb{R}$:
\begin{lemma}\label{lem:iso}
 For all $f \in H^\sigma$ and $r,t \in \mathbb{R}$ we have that
\begin{equation}\label{uni}
\Vert U(t,r) f \Vert_\sigma = \Vert f \Vert_\sigma, \quad \Vert S(t) f \Vert_\sigma = \Vert f \Vert_\sigma.
\end{equation}
\end{lemma}
\begin{proof}
The assertion follows by the definition of the $H^\sigma$ norm together with the relation
\[
\mathrm{e}^{i g(t)\partial_x^2} f(x) = \sum_{k \in \mathbb{Z}^d} \mathrm{e}^{- i g(t)|k|^2} \hat{f}_k \mathrm{e}^{i k\cdot x}
\]
\kat which holds \black for all $f\in   L^2$ such that $f(x) = \sum_{k \in \mathbb{Z}^d} \hat{f}_k \mathrm{e}^{i k \cdot x}$, $x\in \mathbb{T}^{d}$. 
\end{proof}

We point out that a well-posedness theory of \eqref{mnls} under the generality assumed in the present manuscript is very challenging and remains an open problem. Partial results were given in \cite{GC15} covering for instance the case of a fractional Brownian motion $g$ and $d=1$. Our numerical study may help to better understand the properties of \eqref{mnls} for a general class of modulations $g$ with  fractional Sobolev regularity \black and might set a starting point for further analytical investigations.

%under rather general assumptions on $g$ was developed in \cite{GC15}. Remarkably, it was shown that under certain hypothesis of ``irregularity'' (the so-called $(\rho,\gamma)$-irregularity, see \cite[Definition 1.2]{GC15}) of the perturbation
%$g$, \eqref{mnls} can be solved in the usual scale of $H^s$ spaces. Without going into details, let us only point out that the theory of \cite{GC15} covers a wide class of modulations $g$ and in particular also the fractional Brownian motion. The main well-posedness result \cite[Theorem 1.8]{GC15} then states that the modulated NLS \eqref{mnls} in $d=1$ possesses a global solution in $H^s$, for $s\geq0$, and uniqueness holds in a smaller class.
%
%Throughout the remainder of the paper, we assume that $u$ is a unique solution to \eqref{mnls} which belongs to $C([0,T];H^{\sigma+2})$ for $\sigma>d/2$.

Our  goal is to derive a numerical scheme for the mild solution of \eqref{mnls} given by Duhamel's formula
\begin{equation}\label{eq:mild}
 u (t) = U (t, 0) u_0 - i \int_0^t U (t, r) [| u (r) |^2 u (r)] d r,\qquad
 t\in\mathbb{R}.
\end{equation}
Using the flow property $U(t,\xi)=U(t,r)U(r,\xi)$ valid for all $\xi\leq r\leq t$, we deduce that it  satisfies
\[ u(t) = U (t, r) u (r) - i \int_r^t U (t, \xi) [| u (\xi) |^2 u (\xi)] d \xi. \]
Let us now split $[0, T]$ into an equidistant partition $t_0, \ldots, t_N$, with the mesh
size $\tau = \frac{T}{N}$, for some $N\in\mathbb{N}$, which yields  that
\begin{equation}\label{exact}
\begin{aligned}
u(t_{n+1})  
 = U(t_n+\tau,t_n) u (t_n) - i \int_0^{\tau}
U(t_n+\tau,t_n+r)\left [|u(t_n+r) |^2 u(t_n+r)\right] d r.
   \end{aligned}
   \end{equation}
In order to numerically approximate the above time integral, we first use the approximation
$$
u(t_n+r)\approx U(t_n+r,t_n)u(t_n)
$$  
which is of order $r$ in $H^\sigma$ for solutions $u\in C([0,T];H^\sigma)$:
\begin{lemma}\label{lem:1}
Fix  $\sigma>d/2$. Let $u$ be a mild solution to \eqref{mnls} such that $u\in C([0,T];H^{\sigma})$.
% Let $g\in C^{\alpha}([0,T])$ for some $\alpha\in [0,1]$ and  $u(t) \in H^\sigma$ for $0 \leq t \leq T$.  
Then for all $0 \leq t \leq T$ and $0\leq r\leq \tau$ we have
\[
\left\Vert u(t+r) - U(t+r,t) u(t) \right\Vert_\sigma \lesssim r,
\]
where the proportionality constant depends on $\mathrm{sup}_{0 \leq t \leq T} \Vert u(t)\Vert_\sigma$, but can be chosen uniformly in $r$.
\end{lemma}

\begin{proof}
\kat Note that the mild formulation \eqref{exact} yields that
\[
u(t_n+r)  
 = U(t_n+r,t_n) u (t_n) - i \int_0^{r}
U(t_n+r,t_n+\xi)\left [|u(t_n+\xi) |^2 u(t_n+\xi)\right] d \xi.
\]
Together with the bilinear estimate \eqref{bil} this implies that \black
\begin{align*}
\left\Vert u(t_n+r) - U(t_n+r,t) u(t_n) \right\Vert_\sigma \lesssim \black
\int_0^r \left \Vert u(t_n+\xi)\right\Vert_\sigma^3 d\xi
\end{align*}
\kat which yields the assertion. \black
\end{proof}
The above Lemma allows us to define the approximation
\begin{equation}\label{app1}
u_\ast(t_{n+1}) :=  U(t_n+\tau,t_n) u (t_n) - i \int_0^{\tau}
U(t_n+\tau,t_n+r)f \left(U(t_n+r,t_n)u(t_n)\right)  d r, 
\end{equation}
where we let $ f(z) = \vert z \vert^2 z$ for $z\in\mathbb{C}$.
This approximates the exact solution $u$ at time $t_{n+1}$ with order $\tau^2$: 

\begin{corollary}\label{cor:1}
Fix  $\sigma>d/2$. Let $u$ be a mild solution to \eqref{mnls} such that $u\in C([0,T];H^{\sigma})$.
%Let $g\in C^{\alpha}([0,T])$ for some $\alpha\in [0,1]$ and  $u(t) \in H^\sigma$ for $0 \leq t \leq T$.
Then the following approximation holds
\begin{align}\label{est:app1}
\left \Vert u(t_{n+1}) -  u_\ast(t_{n+1})\right\Vert_\sigma \lesssim \tau^2
\end{align}
where the proportionality constant depends on $\mathrm{sup}_{0 \leq t \leq T} \Vert u(t)\Vert_\sigma$, but  is independent of $\tau$.
\end{corollary}
\begin{proof}
\kat Taking the difference of the exact solution \eqref{exact} and the approximation \eqref{app1} yields with $f(z) = \vert z \vert^2 z$, $z\in\mathbb{C}$, that
\begin{align*}
\Vert u(t_{n+1}) - u_\ast(t_{n+1}) \Vert_\sigma \leq \int_0^\tau
 \left  \Vert 
 U(t_n+\tau,t_n+r) \Big[ f\big(U(t_n+r,t_n)u(t_n)\big) - f\big(u(t_n+r)\big)\Big]
 \right \Vert_\sigma  dr.
\end{align*}
The assertion thus follows by Lemma \ref{lem:1}.
\black
\end{proof}
Thanks to Corollary \ref{cor:1} it remains to find a suitable numerical approximation of \eqref{app1}.
%\begin{align*}
%u_\ast(t_{n+1}) =  U(t_n+\tau,t_n) u (t_n) - i \int_0^{\tau}
%U(t_n+\tau,t_n+r)f \left(U(t_n+r,t_n)u(t_n)\right)  d r.
%\end{align*}
To this end, we consider a stratified Monte Carlo approximation of the resulting time integral. More precisely, let  $(\xi_n)_{n\in\mathbb{N}_0}$ be a sequence of independent identically distributed random variables having the uniform distribution on $[0,1]$. We assume that the sequence $(\xi_n)_{n\in\mathbb{N}_0}$ is defined on some underlying probability space $(\Omega,\mathcal{F},\mathbb{P})$ and we denoted by $\mathbb{E}$ the associated expected value. That is, for every $n\in\mathbb{N}_0$, the mapping  $\xi:(\Omega,\mathcal{F})\to ([0,1],\mathcal{B}([0,1]))$ is measurable and for a measurable function $F:\mathbb{R}\to \mathbb{R}$ it holds
$$
\E F(\xi_n) = \int_\Omega F(\xi_n(\omega))\,d\mathbb{P}(\omega)=\int_0^1 F(\xi)\,d\xi,
$$
where the last equality follows by a change of variables from the fact that  $\xi_n$ is uniformly distributed in $[0,1]$.

Then we approximate as follows
\begin{equation}
\begin{aligned}\label{tay}
&\int_0^{\tau}
U(t_n+\tau,t_n+r)\left [| U(t_n+r,t_n)u(t_n) |^2 U(t_n+r,t_n)u(t_n)\right] d r\\
&\qquad \approx\tau U(t_n+\tau,t_n+\tau\xi_n)\left [| U(t_n+\tau\xi_n,t_n)u(t_n) |^2 U(t_n+\tau\xi_n,t_n)u(t_n)\right].
\end{aligned}
\end{equation}

Plugging the above approximation into \eqref{app1}  \kat motivates us to define  the  \emph{randomized exponential integrator}\black
\begin{equation}\label{scheme00}
\begin{aligned} u^{n + 1}  &= U(t_n+\tau,t_n) u^n 
- i \tau U(t_n+\tau,t_n+\tau\xi_n) \left[| U(t_n+\tau\xi_n,t_n)\kat u^n\black |^2 U(t_n+\tau\xi_n,t_n)\kat u^n\black\right]
 \end{aligned}
  \end{equation}
 \kat approximating the mild solution $u$ at time $t_{n+1}$ of the modulated nonlinear Schr\"odinger equation~\eqref{mnls}. \black

\section{Convergence analysis}
\label{sec:conv}
In this section we carry out the convergence analysis of the randomized exponential integrator~\eqref{scheme0}.  Our main result reads as follows.

\begin{theorem}\label{glob}
Fix  $\sigma>d/2$. Let $g\in   W^{\alpha,2}\cap C^{\varepsilon}$ for some $\alpha\in (0,1)$ and an arbitrarily small $\varepsilon\in (0,1)$. Let $u$ be a mild solution to \eqref{mnls} such that $u\in C([0,T];H^{\sigma+2})$. Then there exits $\tau_0 >0$ such that for all $\tau \leq \tau_0$ it holds true that for all $t_N \leq T$
\begin{align*}
\left(\E\max_{M=0,\dots, N}\left\| u(t_M)-u^M\right\|_\sigma^2\right)^{1/2}&\lesssim\tau^{\min\{1,\alpha+\frac{1}{2}\}},
\end{align*}
where the proportionality constant depends on $T$, $\|g\|_{ W^{\alpha,2}}$,\red $\|g\|_{C^{\varepsilon}}$\black and $\sup_{0 \leq t \leq T} \Vert u(t)\Vert_{{\sigma+2}}$, but is independent of $\tau$ and $N$.
\end{theorem}

\begin{remark}[Classical order of convergence]\label{rem:classic}  Let $g \in   W^{\alpha,2}\cap C^{\varepsilon}$ \black for some $\alpha\in (0,1)$ and an arbitrarily small $\varepsilon\in(0,1)$.  With classical techniques we can readily derive the {classical order of convergence} for the randomized exponential integrator \eqref{scheme00} approximating solutions of the modulated nonlinear Schr\"odinger equation~\eqref{mnls}: Taylor series expansion implies that for all $0 \leq \xi \leq 1$ and $0 \leq r \leq \tau$ it holds that
%\red this is true also for $W^{\alpha,2}$, right??? @kat: this way it is not correct, but we have the additional integral in (11) and we could use the expectation:
%\red
%\begin{align*}
%&\left(\int_{0}^{\tau}\left\Vert \big(S(t_n+r) - S(t_n+\xi_{n} \tau) \big) f \right\Vert_\sigma dr\right)^{2}\\
%&\qquad\leq \tau\sup_{0\leq t\leq \tau}\Vert f(t) \Vert^{2}_{\sigma+2} \int_{0}^{\tau}\big\vert g(t_n+r) - g(t_n+\xi_{n} \tau)\big\vert^{2} dr\\
%&\qquad\lesssim \tau^{2\alpha+1}\sup_{0\leq t\leq \tau}\Vert f(t) \Vert^{2}_{\sigma+2}  \int_{0}^{\tau}\big\vert g(t_n+r) - g(t_n+s)\big\vert^{2} dr\\
%&\qquad\leq\tau^{2\alpha+1}\|g\|^{2}_{\red  W^{\alpha,2}} \sup_{0\leq t\leq \tau} \Vert f \Vert_{\sigma+2}^{2}
%\end{align*}
%\black
\[
\left\Vert \big(S(t_n+r) - S(t_n+\xi \tau) \big) f \right\Vert_\sigma \leq \big\vert g(t_n+r) - g(t_n+\xi \tau)\big\vert \Vert f \Vert_{\sigma+2} \leq \|g\|_{C^{\varepsilon}} \tau^\varepsilon \Vert f \Vert_{\sigma+2}.
\]
Applying the above estimate in \eqref{tay} we observe together with Corollary \ref{cor:1} that the randomized exponential integrator \eqref{scheme00} introduces a local error of order
\begin{equation}\label{classloc}
\tau^{\min\{1+\varepsilon,2\}} = \tau^{1+\varepsilon}
\end{equation}
where the last equality follows as $\varepsilon \leq 1$. The isometric property in Lemma \ref{lem:iso} together with the bilinear estimate \eqref{bil} furthermore allows the stability estimate \begin{equation}\label{classstab}
\begin{aligned} 
\Vert u^{n + 1} \Vert_\sigma &\leq  \Vert u^n \Vert_\sigma + c \tau \Vert u^n\Vert_\sigma
\end{aligned}
\end{equation}
where the constant $c$ depends on $\Vert u^n\Vert_\sigma^2$, but can be chosen independently of $\tau$ and $\xi_n$. Thanks to a Lady Windermere's fan argument (see \cite{HLW}) we obtain by the local error \eqref{classloc} together with the stability estimate \eqref{classstab} that there exists a $\tau_0>0$ such that for all $\tau \leq \tau_0$ the global error bound holds
\begin{equation}\label{class:glob}
\Vert u(t_{n+1}) - u^{n+1}\Vert_\sigma \leq c \tau^\varepsilon \quad\text{with} \quad c = c \left(\mathrm{sup}_{0 \leq t \leq t_n}\Vert u(t)\Vert_{\sigma+2}\right).
\end{equation}
\end{remark}

\begin{remark}[A priori bounds]\label{rem:apriori}  Let $g\in   W^{\alpha,2}\cap C^{\varepsilon}$ with $\alpha\in (0,1)$ and arbitrarily small $\varepsilon\in (0,1)$. The classical global error bound \eqref{class:glob} in particular implies the a priori boundedness of $\Vert u^n\Vert_\sigma$ for  all $(\xi_n)_{n\in\mathbb{N}_0}$ in $[0,1]$ and solutions $u \in C([0,T];H^{\sigma+2})$.
\end{remark}

\red
\begin{remark}\label{rem:eps}
We point out that the additional assumption $g\in C^{\varepsilon}$ for some $\varepsilon\in (0,1)$ was only necessary in order to obtain the a priori boundedness of the numerical solution. In other words, the maximal step size $\tau_{0}$ depends on $\varepsilon$ and  the proportionality constant in our main error bound in Theorem \ref{glob} depends on $\|g\|_{C^{\varepsilon}}$. Nevertheless, the established order of convergence is $\alpha+1/2$ independently of the value of $\varepsilon$.
\end{remark}
\black

In Proposition \ref{prop:1} below we will present the essential estimate needed for the proof \kat of \black the main error result Theorem \ref{glob}. It is based on the following discrete version of the Burkholder-Davis-Gundy  inequality (see \cite{B66}).

\begin{theorem}\label{bdg}
For each $p \in (0,\infty)$ there exist positive constants $c_p$ and $C_p$ such
that for every discrete time $\mathbb{R}^d$-valued martingale $\{X_n;\,n\in\mathbb{N}_0\}$ and for every $n\in\mathbb{N}_0$ we have
$$
c_p \E\langle X\rangle_n^{p/2}\leq \E\max_{k=1,\dots,n}|X_k|^p\leq C_p \E \langle X\rangle_n^{p/2},
$$
where $\langle X\rangle_n= |X_0|^2 + \sum_{k=1}^n |X_k-X_{k-1}|^2$ is the quadratic variation of $\{X_n;\,n\in\mathbb{N}_0\}$.
\end{theorem}

%First, note that thanks to Corollary \ref{cor:1} and the estimate \rmk{there was $\tau^2$ in that corollary, no?}
%\begin{align}\label{estuuast}
%\begin{aligned}
%\Vert u(t_{n+1}) - u^{n+1} \Vert_\sigma& \leq \Vert u(t_{n+1}) - u_\ast(t_{n+1})\Vert_\sigma + \Vert u_\ast(t_{n+1})-u^{n+1}\Vert_\sigma \\
%&\leq c \tau + \Vert u_\ast(t_{n+1})-u^{n+1}\Vert_\sigma 
%\end{aligned}
%\end{align}
% it only remains to give a convergence bound on
%\[
%\Vert u^{n+1}�- u_\ast(t_{n+1})\Vert_\sigma
%\]
%with $u_\ast(t_{n+1})$ defined in \eqref{app1}.

Let us introduce the twisted variable
\begin{align}\label{twist}
v(t):=S(t)^{-1}u(t),\qquad v_\ast(t):=S(t)^{-1}u_\ast(t).
\end{align}
In terms of this new variable, the  approximation $u_\ast$ is given by
\begin{align}\label{vast}
v_*(t_{n+1}) = v (t_n) -i  \int_0^{\tau}
S(t_n+r)^{-1}\left [| S(t_n+r)v(t_n) |^2 S(t_n+r)v(t_n)\right] d r,
\end{align}
whereas the numerical solution $v^n=S(t_n)^{-1}u^n$ satisfies
\begin{equation}\label{scheme0}
\begin{aligned} 
v^{n + 1}  &:= v^n 
- i \tau S(t_n+\tau\xi_n)^{-1} \left[| S(t_n+\tau\xi_n)v^n |^2 S(t_n+\tau\xi_n)v^n\right].
 \end{aligned}
  \end{equation}

%
%\begin{remark}
%Due to the $\mathcal{O}(\tau^2)$ error we can only get convergence up to order one (globally), more precisely, our global error will be of order
%\[
%\tau^{\mathrm{min}(1,\alpha+1/2)}.
%\]
%\end{remark}

The following result does not require the additional assumption $g\in C^{\varepsilon}$.

\begin{proposition}\label{prop:1}
Let $g\in   W^{\alpha,2}$ for some $\alpha\in (0,1)$. Then it holds true
\begin{align*}
&\E\max_{M=0,\dots,N}\bigg\|\sum_{n=0}^M\int_0^{\tau}
S(t_n+r)^{-1}\left[| S(t_n+r)v(t_n) |^2 S(t_n+r)v(t_n)\right] d r\\
&\qquad -\tau \sum_{n=0}^M S(t_n+\tau\xi_n)^{-1}\left[| S(t_n+\tau\xi_n)v(t_n) |^2 S(t_n+\tau\xi_n)v(t_n)\right]\bigg\|^2_\sigma\\
&\lesssim \tau^{2\alpha+1}\sup_{0\leq t\leq T}\|v(t)\|^6_{\sigma+2},
\end{align*}
where the proportionality constant depends on $T$, $\|g\|_{ W^{\alpha,2}}$ but is independent of $\tau$ and $N$.
\end{proposition}
\begin{proof}
Since $| u |^2 u = u^2 \cdot \bar{u}$, we write
\[ \int_0^{\tau} S(t_n+r)^{-1} \left[|S(t_n+r) v (t_n) |^2 S(t_n+r)v (t_n)\right] d r \]
\[ = \int_0^{\tau} S(t_n+r)^{-1}\left[ (S(t_n+r)v (t_n))^2 \overline{S(t_n+r){v} (t_n)}\right] d r=:I(\tau). \]
%\[ = \int_0^{\tau} S(t_n+r)^{-1}\left[ (S(t_n+r) v (t_n))^2 \overline{S(t_n+r) {v} (t_n)}\right] d r \]
We intend to estimate the error of the stratified Monte Carlo approximation of the above integral $I(\tau)$ introduced above. Namely, in view of the discussion in Section~\ref{sec:deriv}, it is approximated by
\begin{align*}
I(\tau)&\approx  \tau S(t_n+\tau\xi_n)^{-1}\left [| S(t_n+\tau\xi_n)v(t_n) |^2 S(t_n+\tau\xi_n)v(t_n)\right]\\
&=\tau S(t_n+\tau\xi_n)^{-1}\left [( S(t_n+\tau\xi_n)v(t_n) )^2 \overline{S(t_n+\tau\xi_n){v}(t_n)}\right]=:J(\tau)
\end{align*}
As the next step, we observe that for $x\in \mathbb{T}^{d}$
%\[
%=\int_0^{\tau} S(t_n+r)^{-1}\left[ (S(t_n+r) v (t_n))^2 \overline{S(t_n+r) {v} (t_n)}\right] d r
%\]
\[I(\tau,x) = \sum_{\substack{k, k_1, k_2, k_3 \in \mathbb{Z}^d\\ k = - k_1 + k_2 + k_3}} e^{i k\cdot x}
   \left( \int_0^{\tau} e^{i g (t_n + r) [|k|^2 + |k_1|^2 - |k_2|^2 - |k_3|^2]}
   \hat{v}_{k_2} (t_n) \hat{v}_{k_3} (t_n) \widehat{\bar{v}}_{k_1} (t_n) d r
   \right)  \]
   and similarly
%$$   
% J(\tau)  =\tau S(t_n+\tau\xi_n)^{-1}\left [( S(t_n+\tau\xi_n)v(t_n) )^2 \overline{S(t_n+\tau\xi_n){v}(t_n)}\right]
%$$  
$$
J(\tau,x)= \tau \sum_{\substack{k, k_1, k_2, k_3 \in \mathbb{Z}^d\\ k = - k_1 + k_2 + k_3}} e^{i k \cdot x}
   \left( e^{i g (t_n + \tau\xi_n) [|k|^2 + |k_1|^2 - |k_2|^2 - |k_3|^2]}
   \hat{v}_{k_2} (t_n) \hat{v}_{k_3} (t_n) \widehat{\bar{v}}_{k_1} (t_n) 
   \right) .
$$

For $k,k_{1},k_{2},k_{3}\in\mathbb{Z}^d$ such that $k=-k_{1}+k_{2}+k_{3}$, let us denote $K :=K(k_{1},k_{2},k_{3}):= |k|^2 + |k_1|^2 - |k_2|^2 - |k_3|^2$, where for notational simplicity we omit the dependence of $K$ on its parameters. The same convention will be used in the estimates below. We  denote $\hat{V}_K (t_n) :=
\hat{v}_{k_2} (t_n) \hat{v}_{k_3} (t_n) \kat \overline{\hat{v}}_{k_1} (t_n)$ and define the error
\begin{equation}\label{eq:err}
E^M_K:= \sum_{n = 0}^M \left( \int_0^{\tau} e^{i g (t_n + r) K} -
   e^{i g (t_n + \tau \xi_n) K} d r \hat{V}_K (t_n) \right) ,\quad M=0,\dots,N.
\end{equation} 
  Here
     $$
\tau \sum_{n=0}^M  e^{i g (t_n + \tau \xi_n) K} 
   $$
appearing in the second summand of $E^M_K$  can be regarded as a randomized Riemann sum approximation of the integral
  $$
  \int_0^{M\tau}e^{i g(t_n+r)K} dr .
  $$
%Next, we observe that for every $M=0,\dots,N$ and every $K$ defined as above, we have
%\begin{align*}
%\E[E^{M}_{K}]&=\sum_{n = 0}^M \left( \int_0^{\tau} e^{i g (t_n + r) K} d r-
%  \tau\E( e^{i g (t_n + \tau \xi_n) K} )  \right)\hat{V}_K (t_n)\\
%  &=\sum_{n = 0}^M \left( \int_0^{\tau} e^{i g (t_n + r) K} d r-
%  \tau\int_{0}^{1} e^{i g (t_n + \tau \xi) K} d\xi \right)\hat{V}_K (t_n)=0.\\
%\end{align*}

   Next, we will show that for every fixed $K$ defined above, $E^M_K$ defines a discrete martingale with respect to the parameter $ M=0,\dots,N$ and the filtration $(\mathcal{F}_M)_{M=0,\dots,N}$ given by $\mathcal{F}_M:=\sigma (\xi_i;i=0,\dots,M)$. This will then allow us to apply the Burkholder-Davis-Gundy inequality, Theorem \ref{bdg}. As it will be seen below, the martingale property  is a consequence of the way the randomization  $(\xi_n)_{n\in\mathbb{N}_0}$ was chosen.
   Since all $\xi_n$, $n\in\mathbb{N}_0,$  are independent and uniformly distributed in the interval $[0,1]$, it follows that
   \begin{align}\label{eq:2}
   \begin{aligned}
  & \E\left[ \tau e^{i g (t_n + \tau \xi_n) K} \hat{V}_K (t_n)  \right]= \tau \E\left[e^{i g (t_n + \tau \xi_n) K}\right] \hat{V}_K (t_n)\\
   &\quad=\tau \int_0^1 e^{i g (t_n + \tau \xi) K}d\xi\hat{V}_K (t_n)= \int_0^\tau e^{i g (t_n + r) K}dr\hat{V}_K (t_n).
   \end{aligned}
   \end{align}
   As a consequence
   \begin{align*}
   \E\left[E^M_K\right]&= \sum_{n = 0}^M  \E\left[ \int_0^{\tau} e^{i g (t_n + r) K} -
   e^{i g (t_n + \tau \xi_n) K} d r \hat{V}_K (t_n) \right]\\
   &=\sum_{n = 0}^M   \left(\int_0^{\tau} e^{i g (t_n + r) K}dr- \tau \E\left[  e^{i g (t_n + \tau \xi_n) K} \right]  \right)\hat{V}_K (t_n)=0.
   \end{align*}
   In addition, by definition of $E^M_K$ we deduce that for every $M=0,\dots,N$ the random variable $E^M_K$ is measurable with respect to $\mathcal{F}_M$. Hence the stochastic process $E^M_K$, $M=0,\dots,N,$ is adapted to the filtration $(\mathcal{F}_M)_{M=0,\dots,N}$. To finally verify the martingale property, we let $M\geq m$ and compute the conditional expectation $\E\left[E^M_K|\mathcal{F}_m\right]$. It holds
   \begin{align*}
   \E\left[E^M_K|\mathcal{F}_m\right]&=\sum_{n = 0}^M \E\left[  \int_0^{\tau} e^{i g (t_n + r) K} -
   e^{i g (t_n + \tau \xi_n) K} d r \hat{V}_K (t_n)\Big|\mathcal{F}_m \right]\\
   &=E^m_K+\sum_{n = m+1}^M \E\left[  \int_0^{\tau} e^{i g (t_n + r) K} -
   e^{i g (t_n + \tau \xi_n) K} d r \hat{V}_K (t_n)\Big|\mathcal{F}_m \right]\\
   &=E^m_K+\sum_{n = m+1}^M \left( \int_0^{\tau} e^{i g (t_n + r) K}dr -\tau \E\left[ 
   e^{i g (t_n + \tau \xi_n) K} \big|\mathcal{F}_m \right]\right)\hat{V}_K (t_n)\\
   &=E^m_K+\sum_{n = m+1}^M \left( \int_0^{\tau} e^{i g (t_n + r) K}dr -\tau \E\left[ 
   e^{i g (t_n + \tau \xi_n) K}  \right]\right)\hat{V}_K (t_n)\\
   &=E^m_K,
   \end{align*}
where we used the adaptedness of $E^m_K$, properties of the conditional expectation, independence of $\xi_n$    as well as \eqref{eq:2}. Thus $E^M_K$, $ M=0,\dots,N$, is a martingale with respect to  $(\mathcal{F}_M)_{M=0,\dots,N}$.
   
Hence, we may apply the Parseval identity and the Burkholder-Davis-Gundy inequality, Theorem \ref{bdg}, to obtain
\begin{align*}
&\E \max_{M=0,\dots,N} \left\|x\mapsto  \sum_{\substack{k, k_1, k_2, k_3 \in \mathbb{Z}^d\\ k = - k_1 + k_2 + k_3}}e^{i k \cdot x} E^M_K\right\|_{\sigma}^2\\
& \leq \sum_{k\in\mathbb{Z}^d} (1+|k|^2)^{\sigma}\mathbb{E} \max_{M=0,\dots,N}  \left|
   \sum_{n = 0}^M \sum_{\substack{k_1,k_2,k_3\in\mathbb{Z}^d\\k=-k_1+k_2+k_3}}
   \int_0^{\tau} e^{i g (t_n + r) K^{}} - e^{i g (t_n + \tau \xi_n) K}
   d r \hat{V}_K (t_n)  \right|^2\\
&\lesssim\sum_{k\in\mathbb{Z}^d}(1+|k|^2)^{\sigma} \mathbb{E} \sum_{n = 0}^N  \left|\sum_{\substack{k_1,k_2,k_3\in\mathbb{Z}^d\\k=-k_1+k_2+k_3}}
   \int_0^{\tau} e^{i g (t_n + r) K^{}} - e^{i g (t_n + \tau \xi_n) K^{}} d r
   \hat{V}_K (t_n) \right|^2.
   \end{align*}
%\[ \lesssim \mathbb{E} \sum_{n = 0}^N \sum_{\substack{k, k_1, k_2, k_3 \in \mathbb{Z}\\ k = - k_1 + k_2 + k_3}} \left| 
%   \int_0^{\tau} e^{i g (t_n + r) K^{}} - e^{i g (t_n + \tau \xi_n) K^{}} d r
%   \hat{V}_K (t_n) \right|^2 . \]
Next, we have by the Minkowski integral inequality
\begin{align}\label{eq:err2}
\begin{aligned}
&\left( \mathbb{E}  \left| \sum_{\substack{k_1,k_2,k_3\in\mathbb{Z}^d\\k=-k_1+k_2+k_3}}
   \int_0^{\tau} e^{i g (t_n + r) K^{}} - e^{i g (t_n + \tau \xi_n) K} d r
   \hat{V}_K (t_n) \right|^2 \right)^{1 / 2} \\
  & \quad\leq \sum_{\substack{k_1,k_2,k_3\in\mathbb{Z}^d\\k=-k_1+k_2+k_3}}\left(\E\left| \int_0^\tau  e^{i g (t_n + r) K^{}} - e^{i g (t_n + \tau \xi_n) K}dr\hat{V}_K (t_n) \right|^2\right)^{1/2}\\
  & \quad\leq \tau^{1/2}\sum_{\substack{k_1,k_2,k_3\in\mathbb{Z}^d\\k=-k_1+k_2+k_3}}\left(\E \int_0^\tau \left| e^{i g (t_n + r) K^{}} - e^{i g (t_n + \tau \xi_n) K}\hat{V}_K (t_n) \right|^2dr\right)^{1/2}\\
  & \quad\leq \tau^{1/2}\sum_{\substack{k_1,k_2,k_3\in\mathbb{Z}^d\\k=-k_1+k_2+k_3}}\left(\E \int_0^\tau \left| g (t_n + r)  -  g (t_n + \tau \xi_n) \right|^2dr\right)^{1/2}|K||\hat{V}_K (t_n)|\\
  & \quad= \sum_{\substack{k_1,k_2,k_3\in\mathbb{Z}^d\\k=-k_1+k_2+k_3}}\left(\int_{0}^{\tau} \int_0^\tau \left| g (t_n + r)  -  g (t_n + s) \right|^2dr\,ds\right)^{1/2}|K||\hat{V}_K (t_n)|\\
  & \quad\leq \tau^{\alpha+1/2}\sum_{\substack{k_1,k_2,k_3\in\mathbb{Z}^d\\k=-k_1+k_2+k_3}}\left(\int_{0}^{\tau} \int_0^\tau \frac{\left| g (t_n + r)  -  g (t_n + s) \right|^2}{|r-s|^{2\alpha+1}}dr\,ds\right)^{1/2}|K||\hat{V}_K (t_n)|\\
   & \quad\leq \tau^{\alpha+1/2}\|g\|_{ W^{\alpha,2}(t_{n},t_{n}+\tau)}\sum_{\substack{k_1,k_2,k_3\in\mathbb{Z}^d\\k=-k_1+k_2+k_3}} |K||\hat{V}_K (t_n) |.
%   \\
%     & \quad\lesssim \tau^{\alpha+1} \sum_{\substack{k_1,k_2,k_3\in\mathbb{Z}^d\\k=-k_1+k_2+k_3}}  |K||\hat{V}_K (t_n) |,
     \end{aligned}
\end{align}
Therefore the final error is estimated using the bilinear estimate \eqref{bil} together with the fact that $K = |k|^2 + |k_1|^2 - |k_2|^2 - |k_3|^2 = 2 (k - k_2) \cdot (k - k_3)$ as follows
\begin{align*}
&\E \max_{M=0,\dots,N} \left\|x\mapsto \sum_{\substack{k, k_1, k_2, k_3 \in \mathbb{Z}^d\\ k = - k_1 + k_2 + k_3}}e^{i k\cdot  x} E^M_K\right\|_{\sigma}^2\\
&\quad\lesssim \tau^{2(\alpha+1)}\sum_{n = 0}^N \|g\|^{2}_{ W^{\alpha,2}(t_{n},t_{n}+\tau)} \sum_{k\in\mathbb{Z}^d}(1+|k|^2)^{\sigma}   \left( \sum_{\substack{k_1,k_2,k_3\in\mathbb{Z}^d\\k=-k_1+k_2+k_3}}  |K||\hat{V}_K (t_n) |\right)^2.\\
%&\quad\lesssim\tau^{2\alpha+1}\int_0^T\|v(t)\|_{H^{\sigma+2}}^6 dt.
   \end{align*}
   Using that
   \begin{align*}
   \sum_{k\in\mathbb{Z}^d} & (1+|k|^2)^{\sigma}   \left( \sum_{\substack{k_1,k_2,k_3\in\mathbb{Z}^d\\k=-k_1+k_2+k_3}}  |K||\hat{V}_K (t_n) |\right)^2 \\ & = 
     \sum_{k\in\mathbb{Z}^d}(1+|k|^2)^{\sigma}   \left( \sum_{\substack{k_1,k_2,k_3\in\mathbb{Z}^d\\k=-k_1+k_2+k_3}}  \vert 2 (k - k_2) \cdot (k - k_3)\vert \vert \overline{\hat{v}}_{k_1}(t_n)\vert \vert \hat{v}_{k_2}(t_n)\vert \vert \hat{v}_{k_3}(t_n)\vert \right)^2\\
   &  \lesssim    \sum_{k\in\mathbb{Z}^d}(1+|k|^2)^{\sigma+2}   \left( \sum_{\substack{k_1,k_2,k_3\in\mathbb{Z}^d\\k=-k_1+k_2+k_3}}  \vert \overline{\hat{v}}_{k_1}(t_n)\vert \vert \hat{v}_{k_2}(t_n)\vert \vert \hat{v}_{k_3}(t_n)\vert \right)^2  \lesssim \Vert v(t_n)\Vert_{\sigma+2}^6
   \end{align*}
   we obtain that
   \begin{align*}
\E \max_{M=0,\dots,N} \left\|x\mapsto \sum_{\substack{k, k_1, k_2, k_3 \in \mathbb{Z}^d\\ k = - k_1 + k_2 + k_3}}e^{i k\cdot  x} E^M_K\right\|_{\sigma}^2&\lesssim \tau^{2\alpha+1}\sup_{0 \leq t \leq T}\Vert v(t)\Vert_{\sigma+2}^6 \sum_{n = 0}^N \|g\|^{2}_{ W^{\alpha,2}(t_{n},t_{n}+\tau)}  \\
&\lesssim \tau^{2\alpha+1} \|g\|^{2}_{ W^{\alpha,2}(0,T)} \sup_{0 \leq t \leq T} \Vert v(t)\Vert_{\sigma+2}^6.
%&\quad\lesssim\tau^{2\alpha+1}\int_0^T\|v(t)\|_{H^{\sigma+2}}^6 dt.
   \end{align*}
 This concludes the proof.
\end{proof}

\begin{remark}\label{rem:12}
Note that in the case of $g$ being a Brownian motion as studied in \cite{BdBD,CD}, $g$ is $\alpha$-H\"older continuous for $\alpha\in(0,1/2)$. Hence our analysis gives convergence rate $\alpha+1/2<1$. On the other hand, exploiting instead  It\^o's isometry 
%the fact that
%$$
%\E_W[|W(t)-W(s)|^2]=|t-s|,\qquad s,t\in [0,\infty),
%$$
together with the independence of the increments in \eqref{eq:err2}, our proof can be refined in order to establish the same order of convergence as in \cite{BdBD,CD}, namely 1. Moreover, this would lead to measuring the error in expectation and not only in probability.
\end{remark}

Combining Corollary \ref{cor:1} with Proposition \ref{prop:1}, we obtain the estimate for the global error in Theorem \ref{glob}.

\begin{proof}[Proof of Theorem \ref{glob}]
In the following we will derive the bound
\begin{equation}\label{boundV}
\left(\E\max_{M=0,\dots, N}\left\| v(t_M)-v^M\right\|_\sigma^2\right)^{1/2}\lesssim \tau^{\min\{1,\alpha+\frac{1}{2}\}}.
\end{equation}
The corresponding bound on $u(t_M)-u^M$ then follows by Corollary \ref{cor:1} together with the isometric property \eqref{uni}.

In the following we denote the error $v(t_n)-v^n$ by $e_n$, $n=0,\dots,N,$ and set
$
f(z) := \vert z\vert^2 z.
$
\kat The error in $v$ reads
\begin{align*}
e_{n+1}  :&= v(t_{n+1})-v^{n+1} = (v(t_{n+1})-v_*(t_{n+1}))+(v_*(t_{n+1})-v^{n+1}) \\
& = (v(t_{n+1})-v_*(t_{n+1}))+  (v(t_n) - v^n)\\
& \quad-i  \int_0^{\tau}
S(t_n+r)^{-1} f \left( S(t_n+r)v(t_n)\right) d r + i \tau S(t_n+\tau\xi_n)^{-1} f\left( S(t_n+\tau\xi_n)v^n\right),
\end{align*}
\kat where the  last equality is obtained by subtracting \eqref{scheme0} from \eqref{vast}. 

%Use that
%\[
%U(t,\sigma) = S(t) S(\sigma)^{-1}.
%\]
%Hence,
%
%\begin{align}
%e_{n+1} & := v(t_{n+1})-v^{n+1}  =  \left[ v(t_n) - v^n\right]\\
%& -i S(t_n+\tau)\Big[  \int_0^{\tau}
%S^{-1}(t_n+r) f \left( U(t_n+r,t_n)u(t_n)\right) d r \\&- \tau S^{-1}(t_n+\tau\xi_n) f\left( U(t_n+\tau\xi_n,t_n)u^n\right)\Big].
%\end{align}
Next we use that
\[
f(z) = f(w) + (\overline{z-w}) z^2 + \overline{w}(z+w)(z-w),\qquad z,w\in\mathbb{C}.
\]
The above relation with $z = S(t_n+\tau\xi_n)v^n$ and $w = S(t_n+\tau\xi_n) v(t_n)$ yields that
\begin{align*}
e_{n+1} & = (v(t_{n+1})-v_*(t_{n+1}))+ e_n\\
&\quad -i \Big[  \int_0^{\tau}
S^{-1}(t_n+r) f \left( S(t_n+r)v(t_n)\right) d r \\&
\quad- \tau S^{-1}(t_n+\tau\xi_n) \Big\{
f\left( S(t_n+\tau\xi_n)v(t_n)\right) + \overline{S(t_n+\tau\xi_n)e_n}  z^2 + \overline{w}(z+w)S(t_n+\tau\xi_n) e_n
\Big\}\Big]\\
& = (v(t_{n+1})-v_*(t_{n+1}))+ e_n\\
 &\quad-i \Big[  \int_0^{\tau}
S^{-1}(t_n+r) f \left( S(t_n+r)v(t_n)\right) d r 
- \tau S^{-1}(t_n+\tau\xi_n) 
f\left( S(t_n+\tau\xi_n)v(t_n)\right)\Big]\\
&\quad + \tau \mathcal{B}_n e_n\\
& = (v(t_{n+1})-v_*(t_{n+1}))+ e_n-i L(v(t_n)) + \tau \mathcal{B}_n e_n,
\end{align*}
where
\[
L(v(t_n)) := \Big[  \int_0^{\tau}
S^{-1}(t_n+r) f \left( S(t_n+r)v(t_n)\right) d r 
- \tau S^{-1}(t_n+\tau\xi_n) 
f\left( S(t_n+\tau\xi_n)v(t_n)\right)\Big]
\]
and $\mathcal{B}_n$ denotes some bounded operator, depending on $v^n$, $v(t_n)$ but bounded uniformly in $\omega,\tau,n$ due to Remark \ref{rem:apriori}.
%) and thus ``by induction'' (assuming that $e_n$ is of order $\tau^{\alpha+1/2}$) we have locally $\tau \mathcal{B}\cdot e_n = \mathcal{O}(\tau^{\alpha+1/2+1})$.
%The above formula for $e_{n+1}$ immediately yields a formula for $e_n$ (replacing $n+1$ by $n$), i.e.,
%\begin{align*}
%\textcolor{red}{e_{n}}&  =  e_{n-1} -i L(v(t_{n-1})) + \tau \mathcal{B}_{n-1} e_{n-1}.
%\end{align*}
%
%
Iterating the above formula yields
\begin{align*}
e_M&=(v(t_{M})-v_*(t_{M}))+(v(t_{M-1})-v_*(t_{M-1}))+e_{M-2}-iL(v(t_{M-2}))+\tau\mathcal{B}_{M-2} e_{M-2}\\
&\quad-iLv(t_{M-1})+\tau\mathcal{B}_{M-1} e_{M-1}\\
&=\sum_{k=0}^{M-1}(v(t_{M-k})-v_*(t_{M-k}))+e_0-i\sum_{k=1}^ML(v(t_{M-k}))+\tau\sum_{k=1}^M\mathcal{B}_{M-k}e_{M-k}\\
&=\sum_{n=1}^{M}(v(t_{n})-v_*(t_{n}))+e_0-i\sum_{n=0}^{M-1}L(v(t_{n}))+\tau\sum_{n=0}^{M-1}\mathcal{B}_{n}e_{n}.
\end{align*}
Hence we estimate (using Corollary \ref{cor:1}, definition of the twisted variables \eqref{twist}, the isometric property \eqref{uni} and the fact that $e_0=0$)
\begin{align}\label{eq:r}
\begin{aligned}
&\E\max_{M=1,\dots, N}\left\|e_M\right\|_\sigma^2\lesssim \max_{M=1,\dots, N}\left\|\sum_{n=1}^{M}(v(t_{n})-v_*(t_{n}))\right\|_\sigma^2\\
&\quad+\E\max_{M=1,\dots, N}\left\|\sum_{n=0}^{M-1}L(v(t_{n}))\right\|_\sigma^2+\tau^2 \E\max_{M=1,\dots, N}\left\|\sum_{n=0}^{M-1}\mathcal{B}_{n}e_{n}\right\|_\sigma^2.
\end{aligned}
\end{align}
Note that the first term on the right hand side does not depend on $(\xi_n)$ and hence is independent of $\omega$. According to Corollary \ref{cor:1}, definition of the twisted variables \eqref{twist} and the isometric property \eqref{uni} it can be estimated by
\begin{align*}
\max_{M=1,\dots, N}\left\|\sum_{n=1}^{M}(v(t_{n})-v_*(t_{n}))\right\|_\sigma^2&\leq\left(\sum_{n=1}^{N}\left\|(v(t_{n})-v_*(t_{n}))\right\|_\sigma\right)^2\lesssim \tau^2.
\end{align*}
The last term on the right hand side of \eqref{eq:r} can be estimated as follows
\begin{align*}
\tau^2 \E\max_{M=1,\dots, N}\left\|\sum_{n=0}^{M-1}\mathcal{B}_{n}e_{n}\right\|_\sigma^2&\leq \tau^2\E\left(\sum_{n=0}^{N-1}\left\|\mathcal{B}_{n}e_{n}\right\|_\sigma\right)^2\lesssim \tau^2\E\left(\sum_{n=0}^{N-1}\left\|e_{n}\right\|_\sigma\right)^2\\
&\leq \tau^2 N \E\sum_{n=0}^{N-1}\left\|e_{n}\right\|^2_\sigma\lesssim \tau  \sum_{n=0}^{N-1}\E\max_{M=0,\dots,n}\left\|e_{M}\right\|^2_\sigma.
\end{align*}
Therefore, in view of Proposition 3.4 we estimate the second term on the right hand side of \eqref{eq:r} and obtain 
\begin{align*}
\E\max_{M=0,\dots, N}\left\|e_M\right\|_\sigma^2&\lesssim \tau^2+\tau^{2\alpha+1}+\tau  \sum_{n=0}^{N-1}\E\max_{M=0,\dots,n}\left\|e_{M}\right\|^2_\sigma.
\end{align*}
Finally, we apply the discrete Gronwall Lemma to deduce
\begin{align}\label{ev1}
\E\max_{M=0,\dots, N}\left\|e_M\right\|_\sigma^2&\lesssim \tau^2+\tau^{2\alpha+1},
\end{align}
where the proportional constant depends on $T$ but is independent of $\tau$ and $N$. This implies the estimate \eqref{boundV} as long as $\Vert v^n\Vert_\sigma$ is bounded.

Recall the twisted variables (cf. \eqref{twist})
$$
v(t) = S(t)^{-1}u(t), \quad v^n = S(t)^{-1}u^n
$$
such that in particular due to the isometric property  \eqref{uni} we can conclude that
$$
\Vert v(t)\Vert_\sigma = \Vert u(t)\Vert_\sigma, \qquad \Vert v^n\Vert_\sigma = \Vert u^n\Vert_\sigma
$$
\kat which thanks to Remark \ref{rem:apriori} implies the apriori boundedness of $\Vert v^n\Vert_\sigma$. 

Thanks to Lemma \ref{lem:iso} and Corollary \ref{cor:1} we  then in particular obtain that
\begin{align*}
\E\max_{M=0,\dots, N}\left\| u(t_M) - u^M \right\|_\sigma^2&\lesssim \tau ^2+\tau^{2\alpha+1}
\end{align*} 
which concludes the proof.
\end{proof}

\section{Numerical experiments} \label{sec:num}

In this section we numerically underline the theoretical convergence result of Theorem \ref{glob}.  Furthermore, we compare the convergence behavior of our newly derived randomized exponential integrator \eqref{scheme00} with a classical Strang splitting and exponential integration scheme. For the latter we refer to  \cite{CD,Faou12,HochOst10,Lubich08} and the references therein. 

The numerical experiments emphasize the favorable error behavior of our newly derived scheme over  classical integration methods in the presence of a non-smooth modulation $g$:  Whereas the error of the classical schemes oscillates widely, reaching  large errors, our proposed randomized exponential integrator \eqref{scheme00} averages these oscillations and maintains its convergence rate at low regularity allowing smaller  and in particular reliable errors without any oscillations.

 In our numerical experiments we choose the   classical exponential integrator
\begin{equation}\label{expint}
u^{n + 1}_E  = U(t_n+\tau,t_n) u^n_E
- i \tau U(t_n+\tau,t_n) \left(| u^n_E |^2  u^n_E\right)
\end{equation}
and Strang splitting scheme\begin{equation}\label{Strang}
\begin{aligned}
\psi^{n+1/2}_{-}&=e^{-i\frac{\tau}{2}|\psi^{n+1}|^2 }\psi^{n+1}\\
\psi^{n+1/2}_{+}&=e^{i\left(g\left(t_{n+1}\right)-g\left(t_n\right)\right)\partial_x^2}\psi^{n+1/2}_{-}\\
\psi^{n+1} &=e^{-i\frac{\tau}{2}|\psi^{n+1/2}_{+}|^2}\psi^{n+1/2}_{+}
\end{aligned}
\end{equation}
associated to the modulated Schr\"odinger equation \eqref{mnls}. In all numerical experiments we choose the initial value
\[
u_0(x)= \frac{\mathrm{cos}(x)}{2-\mathrm{sin}(x)}
\]
and use a standard Fourier pseudospectral method for the space discretization where we choose the largest Fourier mode $K = 2^{7}$ (i.e., the spatial mesh size $\Delta x =  0.049$). 
In the sequel, we use the notation $g$ in $W^{\alpha-,2}$ if $g \in W^{\gamma,2}$ for every $\gamma < \alpha$.
To simulate   modulations $g\in   W^{\alpha-,2}$\black, we first choose uniformly distributed random numbers in the interval $[-1,1]$, from which we take the discrete Fourier transform. These Fourier coefficients are then divided by $(1+|k|)^{\alpha +\frac{1}{2}}$ for $k=-\frac{N}{2},...,\frac{N}{2}-1$ and transformed back with the inverse Fourier transform and {normalized}. In Figure \ref{fig_g} below we illustrate the behavior of $g$ in  two cases: $g \in W^{\frac{1}{2}-,2}$ and $g\in W^{\frac{1}{10}-,2}$. \black
\begin{figure}[h!]
    \begin{subfigure}[b]{0.48\textwidth}
        \includegraphics[width=1\textwidth]{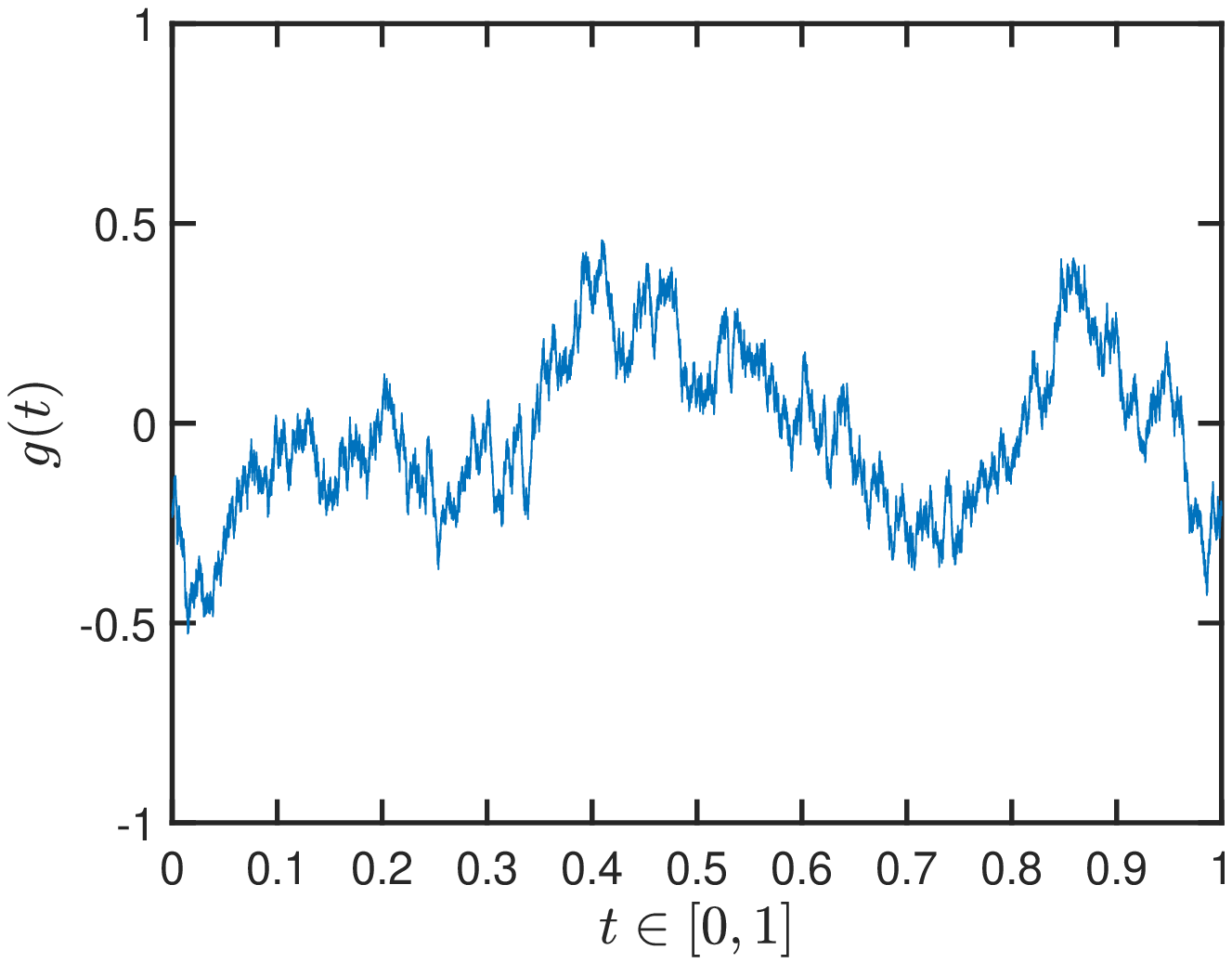}\vskip0.05cm
        \caption{$ g\in W^{\frac{1}{2}-,2}$}
   \end{subfigure}\hfill
       \begin{subfigure}[b]{0.48\textwidth}
        \includegraphics[width=1\textwidth]{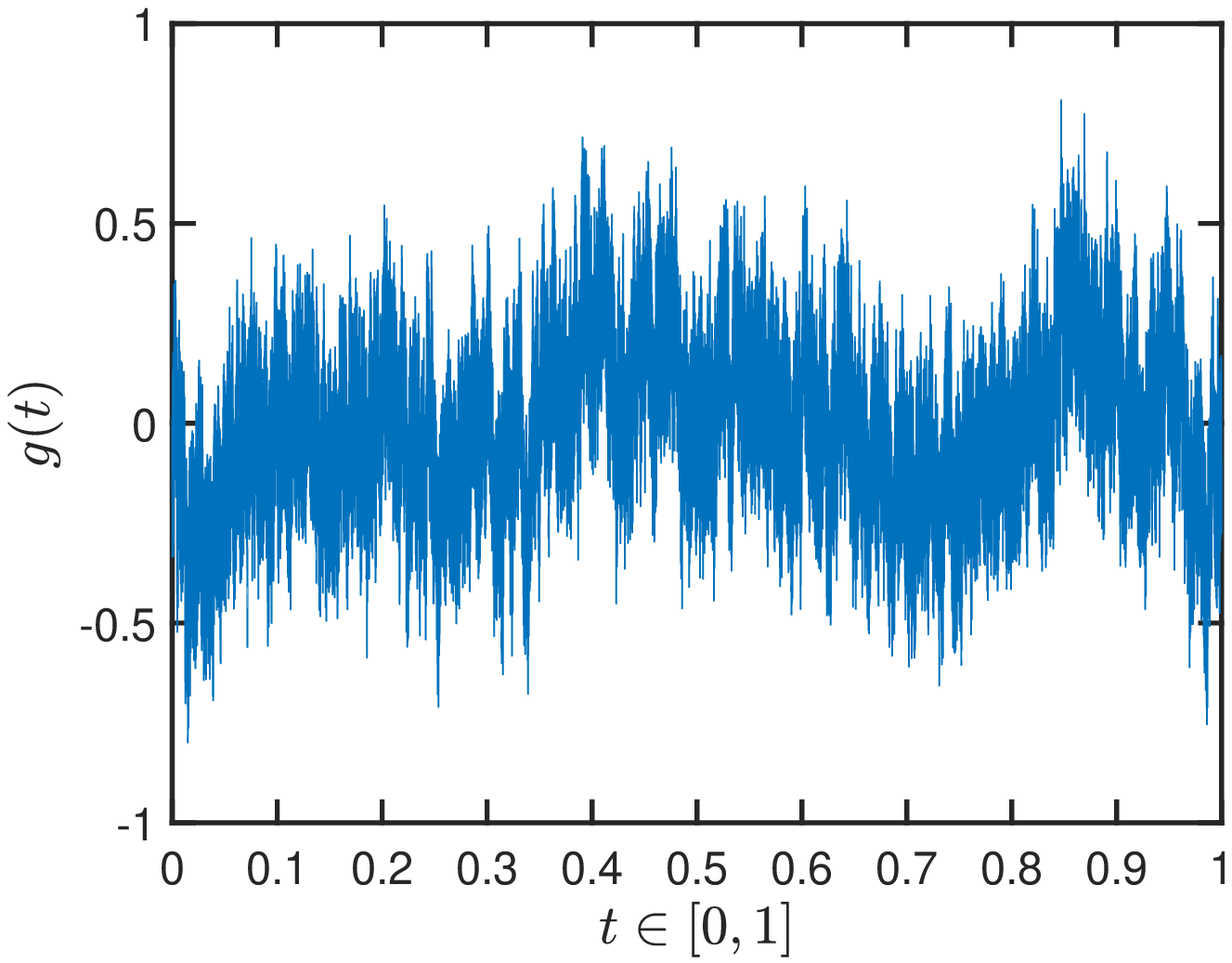}
       \caption{$ g\in W^{\frac{1}{10}-,2}$}
    \end{subfigure}
    \caption{Illustration of the behavior of the modulation function $g\in W^{\alpha-,2}$ in case of $\alpha = 1/2$ and $\alpha = 1/10$.}\label{fig_g}\vskip-0.2cm
   \end{figure}
%. Finally, we normalized these values in such a way that $\Vert g\Vert_{L^2[0,1]}=0.2$ to get the desired discrete initial data in physical space.  
Moreover we consider a smooth example, where $g(t)=\mathrm{sin}(t)$.
In the smooth setting $g(t)= \mathrm{sin}(t)$ we take as a reference solution the Strang splitting scheme with a very small time step size. In the less regular case of a  function $g\in  W^{\alpha-,2}$, i.e., for $\alpha = 1/2$, $\alpha = 1/4$ or $\alpha=1/10$ we take as a reference solution the schemes themselves with a very small time step size.\\
To compute the error $(\E\max_{M=0,\dots, N}\left\|e_M\right\|_1^2)^{1/2}$  (see Theorem \ref{glob}) of our randomized exponential integrator \eqref{scheme00}  we proceed as follows: We denote by  $u^{\mathrm{ref}}(1)$ the reference solution at time $T=1$ and by $u_{(\xi_i)_i^k}(1)$   the approximation computed with the randomized exponential integration scheme \eqref{scheme00}, by using the sequence $(\xi_i)_i^k$. By taking $m \in \mathbb{N}$ sequences, we now use the approximation
\begin{align*}
\left(\E\max_{M=0,\dots, N}\left\|e_M\right\|_1^2\right)^{\frac{1}{2}}\approx \left(\frac{1}{m}\sum_{k=1}^m{\left\|u^{\mathrm{ref}}(1)-u_{(\xi_i)_i^k}(1)\right\|_1^2}\right)^{\frac{1}{2}},
\end{align*}
where the derivative is computed by means of the Fourier transform.  Note that the new scheme can be easily implemented and allows for a parallelization in the sequence.\\
For the Strang splitting and the exponential integrator we compute the classical error in a discrete  $H^1$ norm.

\begin{figure}[h!]
    \begin{subfigure}[b]{0.48\textwidth}
        \includegraphics[width=1\textwidth]{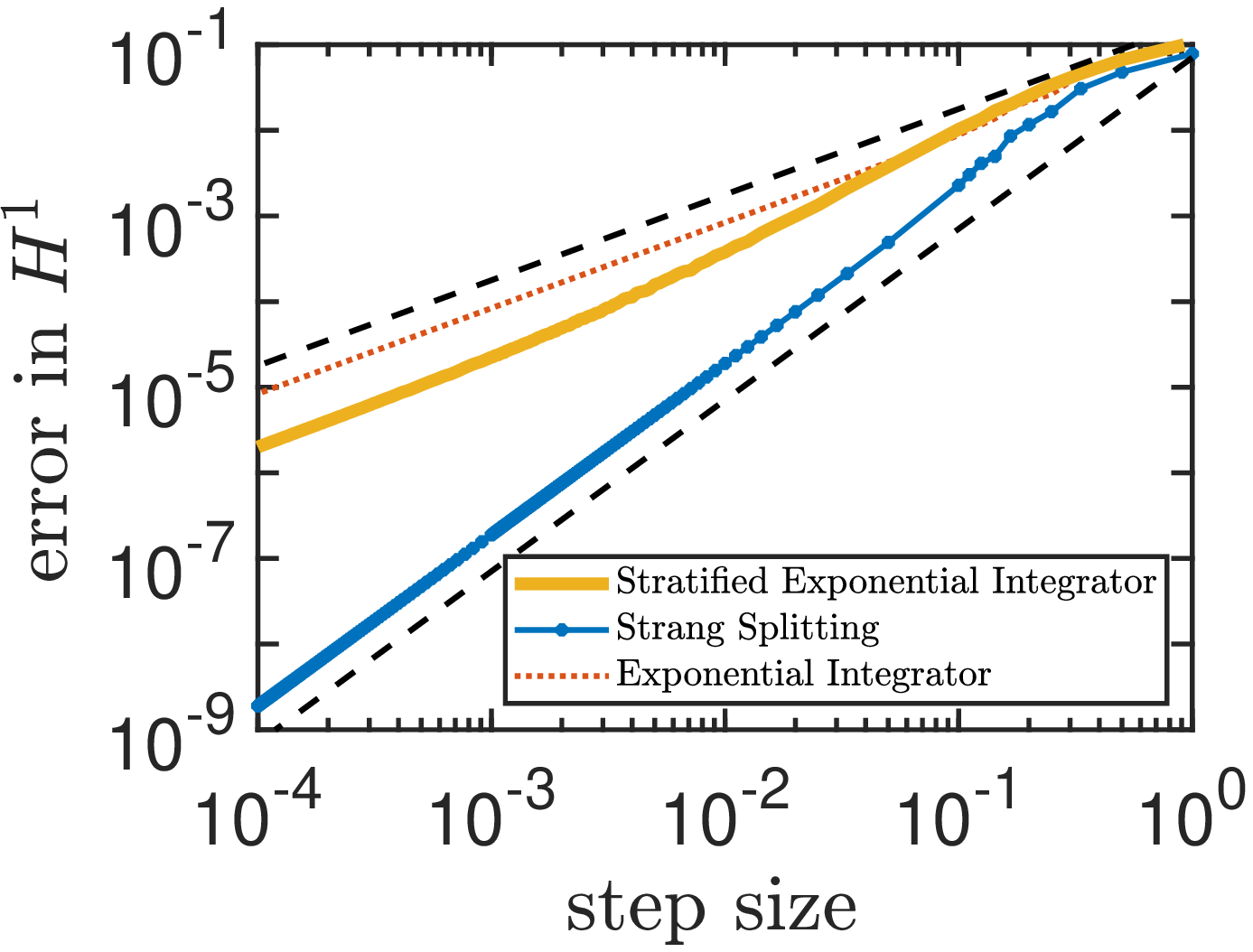}\vskip0.05cm
        \caption{$g(t)=\mathrm{sin}(t)$. The slope of the dashed lines is one and two, respectively.}
   \end{subfigure}\hfill
       \begin{subfigure}[b]{0.48\textwidth}
        \includegraphics[width=1\textwidth]{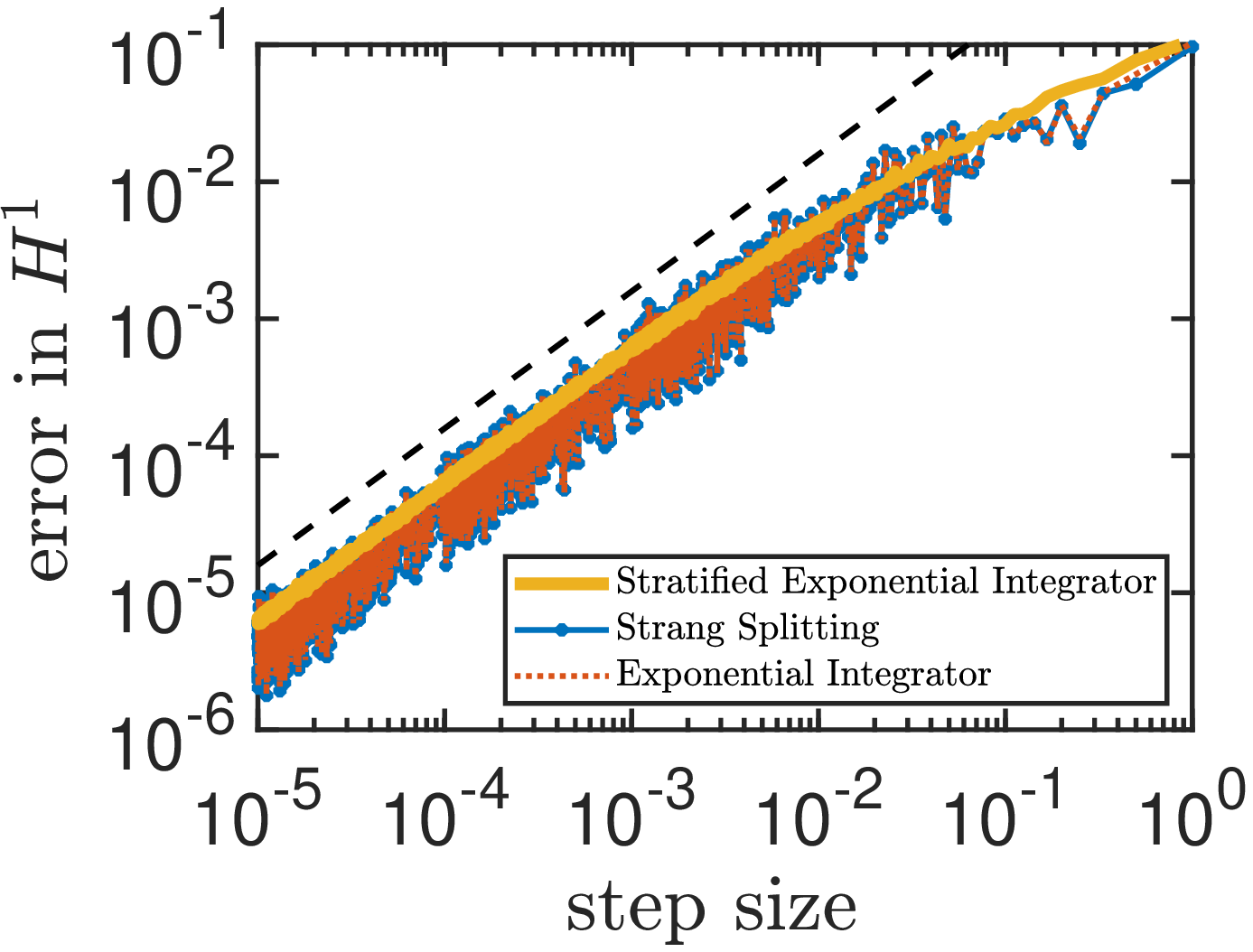}
       \caption{$g\in W^{\frac{1}{2}-,2}$. The slope of the dashed line is $\alpha+1/2 = 1$.}
    \end{subfigure}
    \caption{Convergence plot of the  Strang splitting scheme \eqref{Strang}, the classical exponential integrator \eqref{expint} and the randomized exponential integrator \eqref{scheme00} with $100$ sequences in the case of a smooth function $g$ (left) and a non-regular function $g\in W^{\frac{1}{2}-,2}$  (right).}\label{fig1}\vskip-0.2cm
   \end{figure}

\begin{figure}[h!]
    \begin{subfigure}[b]{0.48\textwidth}
        \includegraphics[width=1\textwidth]{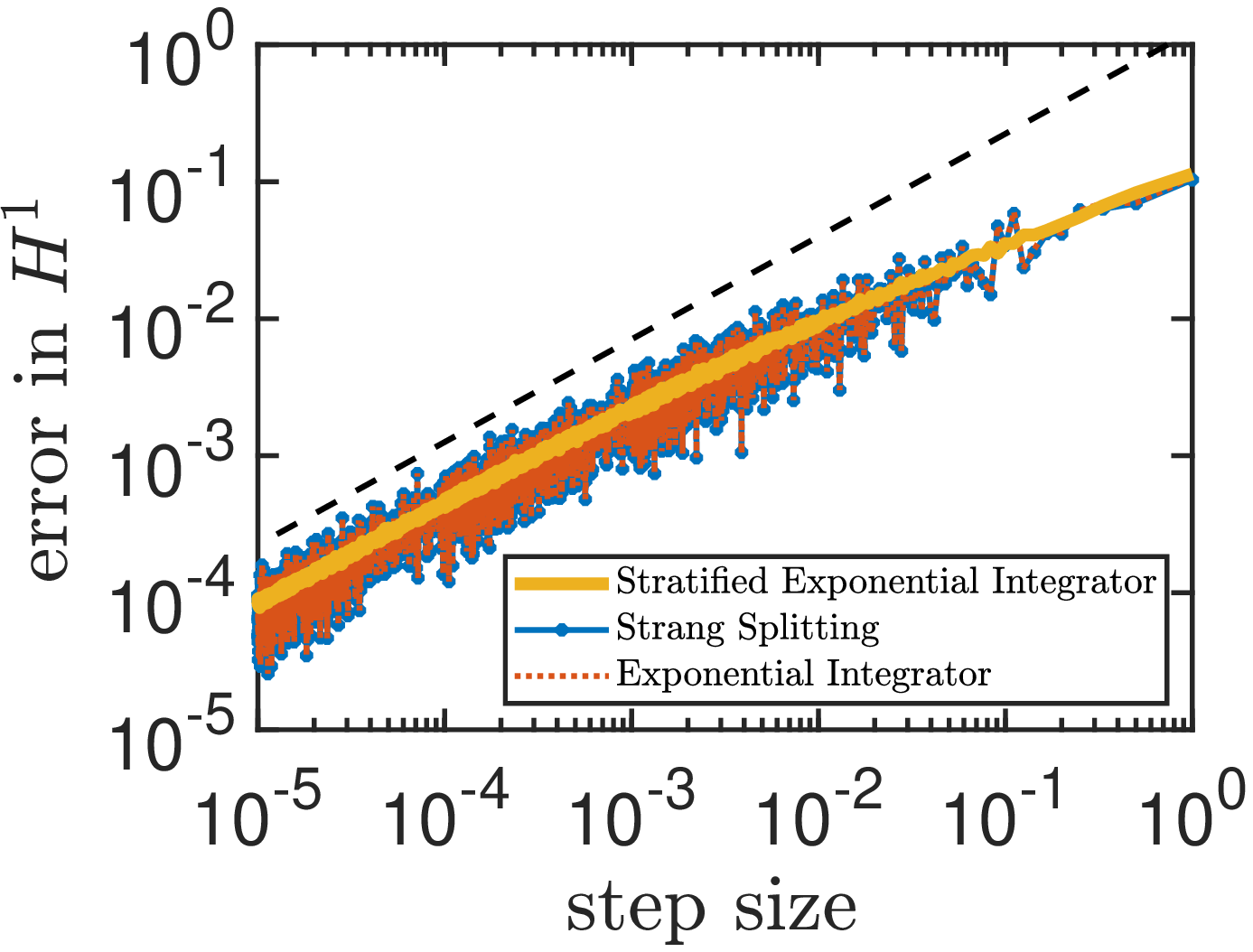}\vskip0.05cm
        \caption{$g\in W^{\frac{1}{4}-,2}$. The slope of the  dashed line is $\alpha+1/2=3/4$.}
   \end{subfigure}\hfill
       \begin{subfigure}[b]{0.48\textwidth}%\vskip-0.1cm
        \includegraphics[width=1\textwidth]{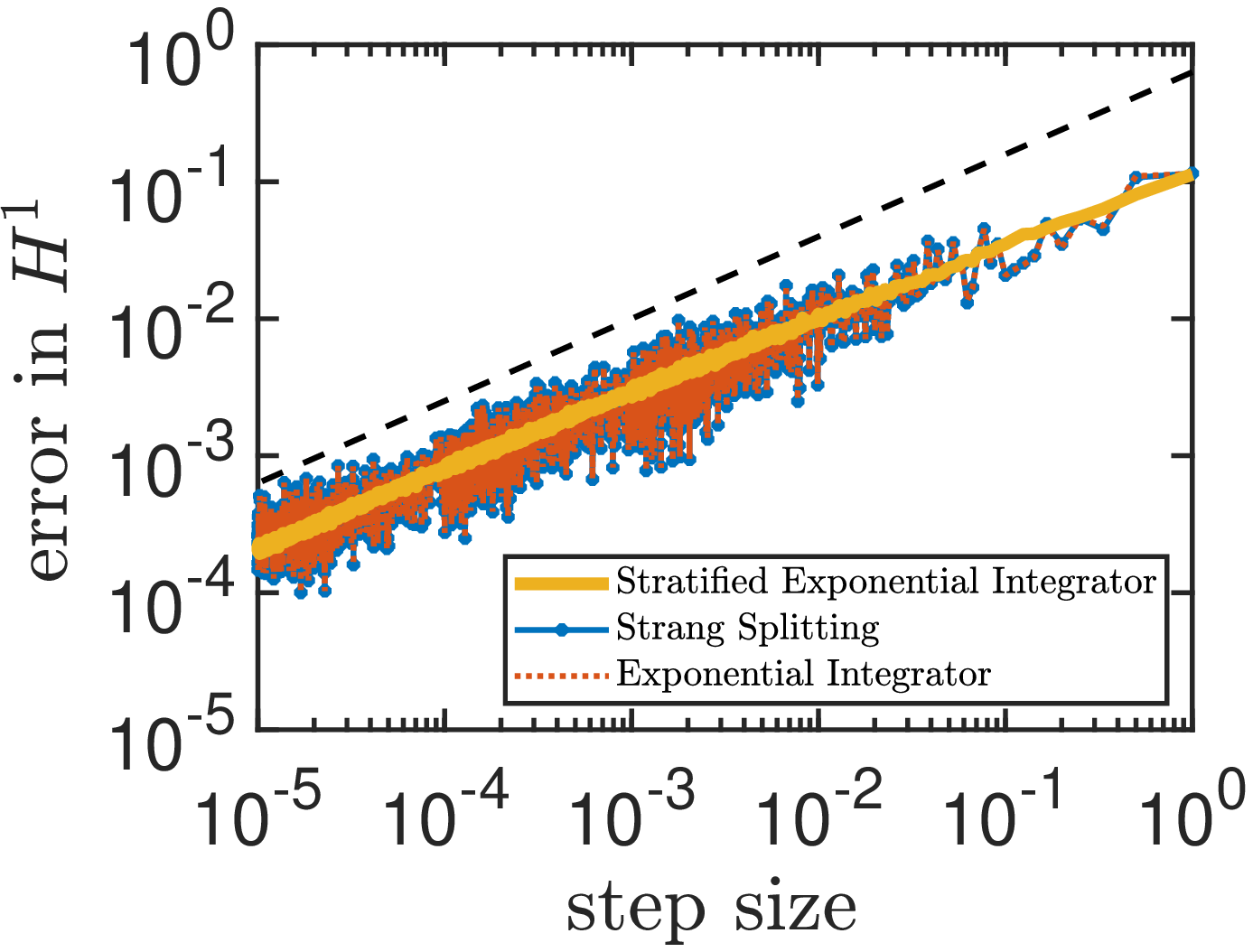}
       \caption{$g\in W^{\frac{1}{10}-,2}$. The slope of the dashed line is $\alpha+1/2=3/5$.}
    \end{subfigure}
    \caption{Convergence plot of the  Strang splitting scheme \eqref{Strang}, the classical exponential integrator \eqref{expint} and the randomized exponential integrator \eqref{scheme00} with $100$ sequences in the case of a function $g\in W^{\alpha-,2}$   with different parameters $\alpha$.}\label{fig1}\vskip-0.2cm
   \end{figure}

%\begin{figure}[h]
%\includegraphics[scale=0.382]{SinPaper1.eps}
%\includegraphics[scale=0.382]{1div2Paper.eps}
%\caption{\textbf{Left picture}: The plot shows the error of the Strang splitting, the classical exponential integrator and the randomized exponential integrator with $100$ sequences for $W_{\infty}(t)=\mathrm{sin}(t)$. The slope of the dashed lines is one and two, respectively. \textbf{Right picture}: $\alpha=1/2$. The plot shows the error of the Strang splitting, the classical exponential integrator and the randomized exponential integrator with $100$ sequences. The slope of the dashed line is one.}
%\end{figure}

%\begin{figure}[h]
%\includegraphics[scale=0.382]{1div4Paper.eps}
%\includegraphics[scale=0.382]{1div10Paper.eps}
%\caption{\textbf{Left picture}: The plot shows the error of the Strang splitting scheme \eqref{Strang}, the classical exponential integrator and the new randomized exponential integrator with $100$ sequences for $H=1/4$. The slope of the  dashed line is $3/4$. \textbf{Right picture}: $H=1/10$. The plot shows the error of the Strang splitting, the classical exponential integrator and the stratified exponential integrator with $100$ sequences. The slope of the  dashed line is $3/5$.}
%\end{figure}

\section*{Acknowledgement}
The authors gratefully acknowledge financial support by the Deutsche Forschungsgemeinschaft (DFG) through CRC 1173.

\end{document}